\documentclass[11pt]{article}
\usepackage{amsfonts,amscd,amssymb, amsmath}
\newtheorem{definition}{Definition}[section]
\newtheorem{lemma}[definition]{Lemma}
\newtheorem{proposition}[definition]{Proposition}
\newtheorem{corollary}[definition]{Corollary}
\newtheorem{remark}[definition]{Remark}

\newtheorem{theorem}[definition]{Theorem}

\def\rawo\lonra{\longrightarrow}

\allowdisplaybreaks[4]

\newenvironment{proof}{{\it Proof.}}{\hfill $ \square $ \vskip 4mm}
\begin{document}
\title{A  quotient of the braid group related  
to pseudosymmetric braided categories
\thanks{Research carried out while the authors were members of the 
CNCSIS project ''Hopf algebras, cyclic homology and monoidal categories''}}
\author{Florin Panaite\\
Institute of Mathematics of the 
Romanian Academy\\ 
PO-Box 1-764, RO-014700 Bucharest, Romania\\
e-mail: Florin.Panaite@imar.ro
\and 
Mihai D. Staic\thanks{Permanent address: Institute of Mathematics of the  
Romanian Academy, 
PO-Box 1-764, RO-014700 Bucharest, Romania.}\\
Department of Mathematics, Indiana University,\\
Rawles Hall, Bloomington, IN 47405, USA\\
e-mail: mstaic@indiana.edu} 
\date{}
\maketitle
\begin{abstract}
Motivated by the recently introduced concept of a {\em pseudosymmetric 
braided monoidal category}, we define the {\em pseudosymmetric group}  
$PS_n$, as the quotient of the braid group $B_n$ by the relations 
$\sigma _i\sigma _{i+1}^{-1}\sigma _i=\sigma _{i+1}\sigma _i^{-1}
\sigma _{i+1}$, with $1\leq i\leq n-2$. It turns out that $PS_n$ is 
isomorphic to the quotient of $B_n$ by the commutator subgroup 
$[P_n, P_n]$ of the pure braid group $P_n$ (which amounts to saying that  
$[P_n, P_n]$ coincides with the normal subgroup of $B_n$ generated by 
the elements $[\sigma _i^2, \sigma _{i+1}^2]$, with $1\leq i\leq n-2$), and  
that $PS_n$ is a linear group.     
\end{abstract}
\section*{Introduction}
${\;\;\;\;}$ 
The notion of {\em symmetric category} is a classical concept in category 
theory. It consists of a monoidal category  
${\cal C}$ equipped with a family of natural isomorphisms 
$c_{X, Y}:X\otimes Y\rightarrow Y\otimes X$ satisfying natural 
``bilinearity'' conditions together with the symmetry relation  
$c_{Y, X}\circ c_{X, Y}=id_{X\otimes Y}$, for all $X, Y\in {\cal C}$. This 
concept was generalized by Joyal and Street in \cite{joyalstreet}, by 
dropping this symmetry condition from the axioms and arriving thus 
at the concept of {\em braided category}, of central importance in 
quantum group theory (cf. \cite{k}, \cite{m}).

Inspired by some recently introduced categorical concepts 
called {\em pure-braided structures} and {\em twines} (cf. \cite{doru} 
and respectively \cite{brug}), in \cite{psvo} was defined the concept of  
{\em pseudosymmetric braiding}, as a generalization of symmetric braidings. 
A braiding $c$ on a strict monoidal category ${\mathcal C}$ is 
pseudosymmetric if it satisfies the following kind of modified 
braid relation:   
\begin{eqnarray*}
(c_{Y, Z}\otimes id_X)\circ (id_Y\otimes c_{Z, X}^{-1})\circ  
(c_{X, Y}\otimes id_Z)
&=&(id_Z\otimes c_{X, Y})\circ (c_{Z, X}^{-1}\otimes id_Y)
\circ (id_X\otimes c_{Y, Z}), 
\end{eqnarray*}
for all $X, Y, Z\in {\mathcal C}$. The main result in \cite{psvo} asserts 
that, if $H$ is a Hopf algebra with bijective antipode, then the 
canonical braiding of the Yetter-Drinfeld category $_H{\cal YD}^H$ 
is pseudosymmetric if and only if $H$ is commutative and cocommutative.

It is well-known that, at several levels, braided categories correspond to 
the braid groups $B_n$, while symmetric categories correspond to the 
symmetric groups $S_n$. It is natural to expect that there exist some groups 
corresponding, in the same way, to pseudosymmetric braided 
categories. Indeed, it is rather clear that these groups, denoted by 
$PS_n$ and called-naturally-the {\em pseudosymmetric groups}, should be the  
quotients of the braid groups $B_n$ by the relations 
$\sigma _i\sigma _{i+1}^{-1}\sigma _i=\sigma _{i+1}\sigma _i^{-1}
\sigma _{i+1}$. Our aim is to investigate and to determine more 
explicitely the structure of these groups. We prove first that the 
kernel of the canonical group morphism $PS_n\rightarrow S_n$ is 
abelian, and consequently $PS_n$ is 
isomorphic to the quotient of $B_n$ by the commutator subgroup 
$[P_n, P_n]$ of the pure braid group $P_n$ (this amounts to saying that 
$[P_n, P_n]$ coincides with the normal subgroup of $B_n$ generated by 
the elements $[\sigma _i^2, \sigma _{i+1}^2]$, with $1\leq i\leq n-2$). 

There exist similarities, but also differences, between braid groups and 
pseudosymmetric groups.   
Bigelow and Krammer proved that braid groups are linear (cf. \cite{big}, 
\cite{krammer}), and we show that so are pseudosymmetric groups. More 
precisely, we prove that the Lawrence-Krammer representation of $B_n$ 
induces a representation of $PS_n$ if the parameter $q$ is chosen to be 
$1$, and that this representation of $PS_n$ is faithful over       
${\mathbb R}[t^{\pm 1}]$. On the other hand, although $PS_n$ is an infinite  
group, like $B_n$, it {\em has} 
nontrivial elements of finite order, unlike $B_n$. 
\section{Preliminaries}\label{sec1}
\setcounter{equation}{0}
${\;\;\;\;}$
We recall briefly the concepts and results that inspired and motivated 
the present paper.
\begin{definition} (\cite{psv}) Let ${\mathcal C}$ be a strict 
monoidal category    
and $T_{X, Y}:X\otimes Y\to X\otimes Y$ a family of natural isomorphisms  
in ${\mathcal C}$.  
We call $T$ a {\bf strong twine}  
if for all $X, Y, Z\in {\mathcal C}$ we have:   
\begin{eqnarray}
&&T_{I,I}=id_I, \label{str1} \\
&&(T_{X, Y}\otimes id_Z)\circ T_{X\otimes Y, Z}=
(id_X\otimes T_{Y, Z})\circ T_{X, Y\otimes Z}, \label{str2} \\
&&(T_{X, Y}\otimes id_Z)\circ (id_X\otimes T_{Y, Z})=
(id_X\otimes T_{Y, Z})\circ (T_{X, Y}\otimes id_Z). \label{str3}
\end{eqnarray}
\end{definition}
\begin{definition} (\cite{psvo}) 
Let ${\mathcal C}$ be a strict monoidal category and $c$ a braiding on  
${\mathcal C}$. We say that $c$ is {\bf pseudosymmetric} if the following   
condition holds, for all $X, Y, Z\in {\mathcal C}$:
\begin{multline}
\;\;\;\;\;\;\;\;\;\;\;\;\;\;\;\;\;\;\;
(c_{Y, Z}\otimes id_X)\circ (id_Y\otimes c_{Z, X}^{-1})\circ  
(c_{X, Y}\otimes id_Z)\\
=(id_Z\otimes c_{X, Y})\circ (c_{Z, X}^{-1}\otimes id_Y)
\circ (id_X\otimes c_{Y, Z}). \label{pseudosymbraiding}
\end{multline}
In this case we say that ${\mathcal C}$ is   
a {\bf pseudosymmetric braided category}.  
\end{definition}
\begin{proposition} (\cite{psvo}) \label{carpseudosym}
Let ${\mathcal C}$ be a strict monoidal category and $c$ a braiding on  
${\mathcal C}$.  
Then the double braiding $T_{X, Y}:=c_{Y, X}\circ c_{X, Y}$ is a strong  
twine if and only if $c$ is pseudosymmetric.  
\end{proposition}
\section{Defining relations for $PS_n$} 
\setcounter{equation}{0}
${\;\;\;\;}$
Let $n\geq 3$ be a natural number. We denote by $B_n$ the braid group on 
$n$ strands, with its usual presentation: generators $\sigma _i$, with 
$1\leq i\leq n-1$, and relations 
\begin{eqnarray}
&&\sigma _i\sigma _j=\sigma _j\sigma _i, \;\;\;if\;\vert i-j\vert \geq 2, 
\label{braid1} \\
&&\sigma _i\sigma _{i+1}\sigma _i=\sigma _{i+1}\sigma _i\sigma _{i+1}, 
\;\;\;if\;1\leq i\leq n-2. \label{braid2}
\end{eqnarray}

We begin with the analogue for braids of Proposition \ref{carpseudosym}:
\begin{proposition}
For all $1\leq i\leq n-2$, the following relations are equivalent in $B_n$: 
\begin{eqnarray}
&&\sigma _i\sigma _{i+1}^{-1}\sigma _i=\sigma _{i+1}\sigma _i^{-1}
\sigma _{i+1}, \label{pseudosym} \\
&&\sigma _i^2\sigma _{i+1}^2=\sigma _{i+1}^2\sigma _i^2. \label{pseudoequiv} 
\end{eqnarray}
\end{proposition}
\begin{proof}
We show first that (\ref{pseudosym}) implies (\ref{pseudoequiv}): 
\begin{eqnarray*}
\sigma _i^2\sigma _{i+1}^2&=&\sigma _i\sigma _{i+1}^{-1}\sigma _{i+1}
\sigma _i\sigma _{i+1}\sigma _{i+1}\\
&\overset{(\ref{braid2})}{=}&\sigma _i\sigma _{i+1}^{-1}\sigma _i
\sigma _{i+1}\sigma _i\sigma _{i+1}\\ 
&\overset{(\ref{braid2}),\;(\ref{pseudosym})}{=}&
\sigma _{i+1}\sigma _i^{-1}\sigma _{i+1}
\sigma _i\sigma _{i+1}\sigma _i\\
&\overset{(\ref{braid2})}{=}&
\sigma _{i+1}\sigma _i^{-1}\sigma _i
\sigma _{i+1}\sigma _i\sigma _i\\
&=&\sigma _{i+1}^2\sigma _i^2.
\end{eqnarray*}
Conversely, we prove that (\ref{pseudoequiv}) implies (\ref{pseudosym}): 
\begin{eqnarray*}
\sigma _i\sigma _{i+1}^{-1}\sigma _i&=&\sigma _i\sigma _{i+1}^{-2}
\sigma _i^{-1}\sigma _i\sigma _{i+1}\sigma _i\\
&\overset{(\ref{braid2})}{=}&\sigma _i\sigma _{i+1}^{-2}
\sigma _i^{-1}\sigma _{i+1}\sigma _i\sigma _{i+1}\\
&=&\sigma _i\sigma _{i+1}^{-2}
\sigma _i^{-2}\sigma _i\sigma _{i+1}\sigma _i\sigma _{i+1}\\
&\overset{(\ref{braid2}),\;(\ref{pseudoequiv})}{=}&\sigma _i
\sigma _i^{-2}\sigma _{i+1}^{-2}\sigma _{i+1}\sigma _i\sigma _{i+1}^2\\
&=&\sigma _i^{-1}\sigma _{i+1}^{-1}\sigma _i\sigma _{i+1}^2\\
&=&\sigma _{i+1}\sigma _{i+1}^{-1}\sigma _i^{-1}\sigma _{i+1}^{-1}
\sigma _i\sigma _{i+1}^2\\
&\overset{(\ref{braid2})}{=}&\sigma _{i+1}
\sigma _i^{-1}\sigma _{i+1}^{-1}\sigma _i^{-1}
\sigma _i\sigma _{i+1}^2\\
&=&\sigma _{i+1}\sigma _i^{-1}\sigma _{i+1},
\end{eqnarray*} 
finishing the proof. 
\end{proof}
\begin{definition}
For a natural number $n\geq 3$, we define the {\bf pseudosymmetric group} 
$PS_n$ as the group with generators $\sigma _i$, with 
$1\leq i\leq n-1$, and relations (\ref{braid1}), (\ref{braid2}) and 
(\ref{pseudosym}), or equivalently (\ref{braid1}), (\ref{braid2}) and 
(\ref{pseudoequiv}). 
\end{definition}
\begin{proposition}
For $1\leq i\leq n-2$ consider the following elements in $PS_n$:
\begin{eqnarray}
&&p_i:=\sigma _i\sigma _{i+1}^{-1}, \;\;\;q_i:=\sigma _i^{-1}\sigma _{i+1}. 
\label{psiq}
\end{eqnarray}
Then the following relations hold in $PS_n$:
\begin{eqnarray}
&&p_i^3=q_i^3=(p_iq_i)^3=1, \;\;\;\forall \;\;1\leq i\leq n-2.
\end{eqnarray}
\end{proposition}
\begin{proof}
The relations $p_i^3=1$ and $q_i^3=1$ follow immediately 
from (\ref{pseudosym}), actually each of them is equivalent to 
(\ref{pseudosym}). Now we compute:   
\begin{eqnarray*}
(p_iq_i)^2&=&(\sigma _i\sigma _{i+1}^{-1}\sigma _i^{-1}\sigma _{i+1})^2\\
&=&\sigma _i\sigma _{i+1}^{-1}\sigma _i^{-1}\sigma _{i+1}
\sigma _i\sigma _{i+1}^{-1}\sigma _i^{-1}\sigma _{i+1}\\
&=&\sigma _i\sigma _{i+1}^{-1}\sigma _i^{-1}\sigma _{i+1}
\sigma _i\sigma _{i+1}\sigma _{i+1}^{-2}\sigma _i^{-1}\sigma _{i+1}\\
&\overset{(\ref{braid2})}{=}&\sigma _i^2\sigma _{i+1}^{-2}\sigma _i^{-1}
\sigma _{i+1}\\
&\overset{(\ref{pseudoequiv})}{=}&\sigma _{i+1}^{-2}\sigma _i\sigma _{i+1}\\
&=&\sigma _{i+1}^{-2}\sigma _i\sigma _{i+1}\sigma _i\sigma _i^{-1}\\
&\overset{(\ref{braid2})}{=}&\sigma _{i+1}^{-1}\sigma _i\sigma _{i+1}
\sigma _i^{-1}\\
&=&(p_iq_i)^{-1},
\end{eqnarray*}
and so we have obtained $(p_iq_i)^3=1$. 
\end{proof} 

Consider now the symmetric group $S_n$ with its usual presentation: 
generators $s_i$, with  
$1\leq i\leq n-1$, and relations (\ref{braid1}), (\ref{braid2}) and 
$s_i^2=1$, for all $1\leq i\leq n-1$. We denote by $\pi :B_n\rightarrow 
S_n$, $\beta :B_n\rightarrow PS_n$, $\alpha :PS_n\rightarrow S_n$, the 
canonical surjective group homomorphisms given by $\pi (\sigma _i)=s_i$, 
$\alpha (\sigma _i)=s_i$, $\beta (\sigma _i)=\sigma _i$, for all 
$1\leq i\leq n-1$. Obviously we have $\pi =\alpha \circ \beta $, hence in 
particular we obtain $Ker (\alpha )=\beta (Ker (\pi ))$. We denote as usual 
$Ker (\pi) =P_n$, the pure braid group on $n$ strands. It is well-known that  
$P_n$ is generated either by the elements   
\begin{eqnarray}
&&a_{ij}:=\sigma_{j-1}\sigma_{j-2}\cdots \sigma_{i+1}\sigma_i^2
\sigma_{i+1}^{-1}\cdots \sigma_{j-2}^{-1}\sigma_{j-1}^{-1}, \;\;\;
1\leq i<j\leq n, \label{aij}
\end{eqnarray}
or by the elements 
\begin{eqnarray}
&&b_{ij}:=\sigma_{j-1}^{-1}\sigma_{j-2}^{-1}\cdots \sigma_{i+1}^{-1}
\sigma_i^2\sigma_{i+1}\cdots \sigma_{j-2}\sigma_{j-1}, 
\;\;\;1\leq i<j\leq n. \label{bij}
\end{eqnarray}
It is easy to see that in $B_n$ we have 
\begin{eqnarray}
&&\sigma _{i+1}\sigma _i^2\sigma _{i+1}^{-1}=\sigma _i^{-1}\sigma _{i+1}^2
\sigma _i, \label{cucu1} \\
&&\sigma _{i+1}^{-1}\sigma _i^2\sigma _{i+1}=\sigma _i\sigma _{i+1}^2
\sigma _i^{-1}, \label{cucu2}
\end{eqnarray}
and by using repeatedly these relations we obtain the following equivalent 
descriptions of the elements $a_{ij}$ and $b_{ij}$:
\begin{eqnarray}
&&a_{ij}=\sigma_i^{-1}\sigma _{i+1}^{-1}\cdots \sigma_{j-2}^{-1}
\sigma_{j-1}^2\sigma_{j-2}\cdots \sigma _{i+1}\sigma_i, \;\;\;
1\leq i<j\leq n,  \label{aijequiv} \\
&&b_{ij}=\sigma_i\sigma_{i+1}\cdots \sigma_{j-2}
\sigma_{j-1}^2\sigma_{j-2}^{-1}\cdots \sigma _{i+1}^{-1}\sigma_i^{-1}, 
\;\;\;1\leq i<j\leq n.  \label{bijequiv}
\end{eqnarray}

Now, for all $1\leq i<j\leq n$, we define $A_{i, j}$ and $B_{i, j}$ as the 
elements in $PS_n$ given by $A_{i, j}:=\beta (a_{ij})$ and 
$B_{i, j}:=\beta (b_{ij})$. From the above discussion it follows that 
$Ker (\alpha )$ is generated by $\{A_{i, j}\}_{1\leq i<j\leq n}$ and also 
by $\{B_{i, j}\}_{1\leq i<j\leq n}$. 
\begin{lemma}
The following relations hold in $PS_n$, for $1\leq i<j<n$: 
\begin{eqnarray}
&&A_{i, j+1}=\sigma _jA_{i, j}\sigma _j^{-1}, \label{slab1} \\
&&B_{i, j+1}=\sigma _j^{-1}B_{i, j}\sigma _j. \label{slab2}
\end{eqnarray}
\end{lemma}
\begin{proof}
These relations are actually consequences of corresponding relations in 
$B_n$ for $a_{ij}$'s and $b_{ij}$'s, which in turn follow immediately by 
using the formulae (\ref{aij}) and (\ref{bij}).   
\end{proof}
\begin{lemma}
The following relations hold in $PS_n$, for all $i, j\in \{1, 2, ..., n\}$ 
with $i+1<j$: 
\begin{eqnarray}
&&A_{i, j}=\sigma _iA_{i+1, j}\sigma _i^{-1}, \label{ajutaij} \\
&&B_{i, j}=\sigma _i^{-1}B_{i+1, j}\sigma _i. \label{ajutbij}
\end{eqnarray}
\end{lemma}
\begin{proof}
We prove (\ref{ajutaij}), while (\ref{ajutbij}) is similar and left to the 
reader. Note that in $PS_n$ we have    
$\sigma_{i+1}^{-1}\sigma_i^2\sigma_{i+1}=
\sigma_{i+1}\sigma_i^2\sigma_{i+1}^{-1}$, 
which together with (\ref{cucu2}) implies 
$\sigma_{i}\sigma_{i+1}^2\sigma_{i}^{-1}=
\sigma_{i+1}\sigma_i^2\sigma_{i+1}^{-1}$, 
hence 
\begin{eqnarray*}
A_{i,j}&=&\sigma_{j-1}\sigma_{j-2}\cdots 
(\sigma_{i+1}\sigma_i^2\sigma_{i+1}^{-1})\cdots 
\sigma_{j-2}^{-1}\sigma_{j-1}^{-1}\\
&=&\sigma_{j-1}\sigma_{j-2}\cdots 
(\sigma_{i}\sigma_{i+1}^2\sigma_{i}^{-1})\cdots 
\sigma_{j-2}^{-1}\sigma_{j-1}^{-1}\\
&=&\sigma _i\sigma_{j-1}\sigma_{j-2}\cdots 
\sigma_{i+1}^2\cdots  
\sigma_{j-2}^{-1}\sigma_{j-1}^{-1}\sigma _i^{-1}\\
&=&\sigma_iA_{i+1,j}\sigma_i^{-1},
\end{eqnarray*} 
finishing the proof.
\end{proof}
\begin{proposition}
For all $1\leq i<j\leq n$, we have $A_{i, j}=B_{i, j}$ in $PS_n$. 
\end{proposition}
\begin{proof}
We use (\ref{ajutaij}) repeatedly:
\begin{eqnarray*}
A_{i,j}&=&\sigma _iA_{i+1, j}\sigma _i^{-1}\\
&=&\sigma _i\sigma _{i+1}A_{i+2, j}\sigma _{i+1}^{-1}\sigma _i^{-1}\\
&&\cdots \\
&=&\sigma _i\sigma _{i+1}\cdots \sigma _{j-2}A_{j-1, j}\sigma _{j-2}^{-1}
\cdots \sigma _{i+1}^{-1}\sigma _i^{-1}\\
&=&\sigma _i\sigma _{i+1}\cdots \sigma _{j-2}\sigma _{j-1}^2
\sigma _{j-2}^{-1}\cdots \sigma _{i+1}^{-1}\sigma _i^{-1}\\
&\overset{(\ref{bijequiv})}{=}&B_{i, j}, 
\end{eqnarray*}
finishing the proof.
\end{proof}
\begin{lemma} For all $1\leq i<j\leq n$ and $1\leq h\leq k<n$  
the following relations hold in $PS_n$: 
\begin{eqnarray}
&&A_{i,j}\sigma _i^2=\sigma _i^2A_{i,j},    \label{aijsigmai2} \\
&&A_{h, k+1}\sigma _k^2=\sigma _k^2A_{h, k+1}. \label{aikuri} 
\end{eqnarray}
\end{lemma}
\begin{proof}
Note first that (\ref{aijsigmai2}) is obvious for $j=i+1$. Assume that 
$i+1<j$;   
using the fact that $A_{r, s}=B_{r, s}$ for all $r, s$ we compute:
\begin{eqnarray*}
A_{i,j}\sigma _i^2&\overset{(\ref{ajutaij})}{=}&\sigma _iA_{i+1, j}
\sigma _i\\
&=&\sigma _iB_{i+1,j}\sigma _i\\
&\overset{(\ref{ajutbij})}{=}&\sigma _i^2B_{i,j}\\
&=&\sigma _i^2A_{i, j}. 
\end{eqnarray*}
Note also that (\ref{aikuri}) is obvious for $h=k$. Assume that $h<k$; 
using again $A_{r, s}=B_{r, s}$ for all $r, s$ we compute:
\begin{eqnarray*}
A_{h, k+1}\sigma _k^2&\overset{(\ref{slab1})}{=}&\sigma _kA_{h, k}\sigma _k\\
&=&\sigma _kB_{h, k}\sigma _k\\
&\overset{(\ref{slab2})}{=}&\sigma _k^2B_{h, k+1}\\
&=&\sigma _k^2A_{h, k+1}, 
\end{eqnarray*}
finishing the proof.
\end{proof}
\section{The structure of $PS_n$} 
\setcounter{equation}{0}
${\;\;\;\;}$
We denote by $\mathfrak{P}_n$ the kernel of the morphism 
$\alpha :PS_n\rightarrow S_n$ defined above.  
\begin{proposition} \label{abel}
$\mathfrak{P}_n$ is an abelian group. 
\end{proposition}
\begin{proof}
It is enough to prove that any two elements $A_{i, j}$ and $A_{k, l}$ 
commute in $PS_n$. We only have to analyze the following seven cases for the 
numbers $i, j, k, l$:\\[2mm]  
(1) $i<j<k<l$: this is trivial (comes from a relation in $P_n$).\\[2mm] 
(2) $i<j=k<l$: we write  
\begin{eqnarray*} 
&&A_{i,j}=\sigma_i^{-1}\sigma _{i+1}^{-1}\cdots \sigma_{j-2}^{-1}
\sigma_{j-1}^2\sigma_{j-2}\cdots \sigma _{i+1}\sigma_i, \\ 
&&A_{j,l}=\sigma_{l-1}\sigma_{l-2}\cdots \sigma_{j+1}\sigma_j^2
\sigma_{j+1}^{-1}\cdots \sigma_{l-2}^{-1}\sigma_{l-1}^{-1}, 
\end{eqnarray*}
and we obtain $A_{i, j}A_{j, l}=A_{j, l}A_{i, j}$ by using (\ref{braid1}) 
and the fact that $\sigma_{j-1}^2$ and $\sigma_j^2$ commute in $PS_n$.\\[2mm] 
(3) $i<k<j<l$: this follows since $A_{k, l}=B_{k, l}$ in $PS_n$, 
and $a_{ij}$ and $b_{kl}$ commute in $P_n$ if $i<k<j<l$ (this is easily seen 
geometrically). \\[2mm]
(4) $i=k<j=l$: this is trivial.\\[2mm]
(5) $i<k<l<j$: this is trivial (comes from a relation in $P_n$). \\[2mm]
(6) $i=k<j<l$: in case $j=i+1$, we have $A_{i, j}=\sigma _i^2$ and so we 
obtain $A_{i, j}A_{i, l}=A_{i, l}A_{i, j}$ by using (\ref{aijsigmai2}); 
assuming now $i+1<j$,  
by using repeatedly (\ref{ajutaij}) we can compute:  
\begin{eqnarray*}
A_{i,j}A_{i,l}&=&\sigma_iA_{i+1,j}A_{i+1,l}\sigma_i^{-1}\\
&=&\sigma _i\sigma _{i+1}A_{i+2, j}A_{i+2, l}\sigma _{i+1}^{-1}
\sigma _i^{-1}\\
&&\cdots \\
&=&\sigma_i\sigma_{i+1}\cdots \sigma_{j-2}A_{j-1,j}A_{j-1,l}
\sigma_{j-2}^{-1}\cdots \sigma_{i+1}^{-1}\sigma_i^{-1}, 
\end{eqnarray*}
and similarly 
\begin{eqnarray*}
&&A_{i, l}A_{i, j}=\sigma _i\sigma _{i+1}\cdots 
\sigma _{j-2}A_{j-1, l}A_{j-1, j}\sigma _{j-2}^{-1}\cdots 
\sigma _{i+1}^{-1}\sigma _i^{-1}, 
\end{eqnarray*}
and these are equal because 
$A_{j-1, j}=\sigma _{j-1}^2$ and by (\ref{aijsigmai2}) we have 
$\sigma _{j-1}^2A_{j-1, l}=A_{j-1, l}\sigma _{j-1}^2$.\\[2mm]  
(7) $i<k<j=l$: in case $j=k+1$, we have $A_{k, j}=\sigma _k^2$ and so we 
obtain $A_{i, j}A_{k, j}=A_{k, j}A_{i, j}$ by using (\ref{aikuri}); 
assuming now $k+1<j$,   
by using repeatedly (\ref{slab1}) we can compute:
\begin{eqnarray*}
A_{i, j}A_{k, j}&=&\sigma _{j-1}A_{i, j-1}A_{k, j-1}\sigma _{j-1}^{-1}\\
&=&\sigma _{j-1}\sigma _{j-2}A_{i, j-2}A_{k, j-2}\sigma _{j-2}^{-1}
\sigma _{j-1}^{-1}\\
&&\cdots \\
&=&\sigma _{j-1}\sigma _{j-2}\cdots \sigma _{k+1}A_{i, k+1}A_{k, k+1}
\sigma _{k+1}^{-1}\cdots \sigma _{j-2}^{-1}\sigma _{j-1}^{-1}, 
\end{eqnarray*}
and similarly 
\begin{eqnarray*}
&&A_{k, j}A_{i, j}=\sigma _{j-1}\sigma _{j-2}\cdots \sigma _{k+1}
A_{k, k+1}A_{i, k+1}\sigma _{k+1}^{-1}\cdots \sigma _{j-2}^{-1}
\sigma _{j-1}^{-1},
\end{eqnarray*}
and these are equal because $A_{k, k+1}=\sigma _k^2$ and by 
(\ref{aikuri}) we have $A_{i, k+1}\sigma _k^2=\sigma _k^2A_{i, k+1}$.   
\end{proof}

Let $G$ be a group. If $x, y\in G$ we denote by $[x, y]:=x^{-1}y^{-1}xy$ 
the commutator of $x$ and $y$, and by $G'$ the commutator subgroup of $G$ 
(the subgroup of $G$ generated by all commutators $[x, y]$), which is the  
smallest normal subgoup $N$ of $G$ with the property that $G/N$ is abelian. 
Moreover, $G'$ is a characteristic subgroup of $G$, i.e. 
$\theta (G')=G'$ for all $\theta \in Aut (G)$. 
\begin{proposition}
$\mathfrak{P}_n\simeq P_n/P_{n}'\simeq {\mathbb Z}^{\frac{n(n-1)}{2}}$.
\end{proposition} 
\begin{proof}
For $1\leq i\leq n-2$ we define the elements $t_i\in P_n$ by 
$t_i:=[\sigma _i^2, \sigma _{i+1}^2]=[a_{i, i+1}, a_{i+1, i+2}]$. These 
elements are the relators added to the ones of $B_n$ in order to obtain 
$PS_n$, and so (see \cite{coxeter}) the kernel of the map 
$\beta :B_n\rightarrow PS_n$ defined above coincides with the normal 
subgroup of $B_n$ generated by $\{t_i\}_{1\leq i\leq n-2}$, which will be 
denoted by $L_n$. We obviously have $L_n\subseteq P_n$, and if we 
consider the map $\beta $ restricted to $P_n$, we have a surjective 
morphism $P_n\rightarrow \mathfrak{P}_n$ with kernel $L_n$, so 
$\mathfrak{P}_n\simeq P_n/L_n$. By Proposition \ref{abel} we know that 
$\mathfrak{P}_n$ is abelian, so we obtain $P_n'\subseteq L_n$. On the 
other hand, since $P_n'$ is characteristic in $P_n$ and $P_n$ is 
normal in $B_n$, it follows (see \cite{suzuki}, Prop. 6.14) 
that $P_n'$ is normal in $B_n$, and since $t_1, \cdots , t_{n-2}\in P_n'$  
and $L_n$ is the normal subgroup of $B_n$ generated by 
$\{t_i\}_{1\leq i\leq n-2}$ we obtain $L_n\subseteq P_n'$. Thus, we 
have obtained $L_n=P_n'$ and so $\mathfrak{P}_n\simeq P_n/P_n'$. 
On the other hand, it is well-known that   
$P_n/P_{n}'\simeq {\mathbb Z}^{\frac{n(n-1)}{2}}$.      
\end{proof}

Note also that as a consequence of the equality $L_n=P_n'$ we obtain: 
\begin{corollary}
$PS_n\simeq B_n/P_n'$.  
\end{corollary}

The extension with abelian kernel $1\rightarrow \mathfrak{P}_n 
\rightarrow PS_n\rightarrow S_n\rightarrow 1$ induces an action of $S_n$ 
on $\mathfrak{P}_n$, given by $\sigma \cdot a=\tilde{\sigma }a
\tilde{\sigma }^{-1}$, for $\sigma \in S_n$ and $a\in \mathfrak{P}_n$, 
where $\tilde{\sigma }$ is an element of $PS_n$ with 
$\alpha (\tilde{\sigma })=\sigma $. In particular, on generators we have 
$s_k\cdot A_{i, j}=\sigma _kA_{i, j}\sigma _k^{-1}$, for 
$1\leq k\leq n-1$ and $1\leq i<j\leq n$. By using some of the formulae 
given above, one can describe explicitely this action, as follows: \\[2mm] 
(1) $s_k\cdot A_{i, j}=A_{i, j}\;$ if $\;k<i-1$; \\
(2) $s_{i-1}\cdot A_{i, j}=A_{i-1,j}$; \\
(3) $s_i\cdot A_{i, j}=A_{i+1,j}\;$ if $j-i>1$,   and 
$s_i\cdot A_{i,i+1}=A_{i,i+1}$; \\
(4) $s_k\cdot A_{i, j}=A_{i, j}\;$ if $\;i<k<j-1$; \\
(5) $s_{j-1}\cdot A_{i, j}=A_{i,j-1}\;$ if $\;j-i>1$,  and  
$s_{j-1}\cdot A_{j-1,j}=A_{j-1,j}$; \\
(6) $s_j\cdot A_{ij}=A_{i,j+1}\;$ for $\;1\leq i<j<n$;  \\
(7) $s_k\cdot A_{i, j}=A_{i, j}\;$ if $\;j<k$.\\[2mm]
Also, one can easily see that these formulae may be expressed more 
compactly as follows: if $\sigma \in \{s_1, \cdots, s_{n-1}\}$ and 
$1\leq i<j\leq n$ then $\sigma \cdot A_{i, j}=A_{\sigma (i), \sigma (j)}$, 
where we made the convention $A_{r, t}:=A_{t, r}$ for $t<r$. Since 
$s_1, \cdots, s_{n-1}$ generate $S_n$, this immediately implies: 
\begin{proposition}
For any $\sigma \in S_n$ and $1\leq i<j\leq n$, the action of $\sigma $ 
on $A_{i, j}$ is given by $\sigma \cdot A_{i, j}=A_{\sigma (i), \sigma (j)}$, 
with the convention $A_{r, t}:=A_{t, r}$ for $t<r$.  
\end{proposition}
\begin{proposition}
In $PS_n$ there is no element of order 2 whose image in $S_n$ 
is the transposition $s_1=(1, 2)$. Consequently, the extension 
$1\rightarrow \mathfrak{P}_n\rightarrow PS_n\rightarrow S_n\rightarrow 1$ 
is {\bf not} split.  
\end{proposition}
\begin{proof} Take $x\in PS_n$ such that $\alpha (x)=s_1$. Since 
$\alpha (\sigma _1)=s_1$, we obtain $x\sigma _1^{-1}\in Ker (\alpha )=
\mathfrak{P}_n$. One can easily see that 
$\mathfrak{P}_n$ is freely generated by 
$\{A_{i, j}\}_{1\leq i<j\leq n}$, so we can write uniquely   
$x=\prod_{1\leq i<j\leq n}A_{i, j}^{m_{ij}}\sigma_1$, with $m_{ij}\in 
{\mathbb Z}$. We compute:  
\begin{eqnarray*}
x^2&=&(\prod_{1\leq i<j\leq n}A_{i, j}^{m_{ij}}\sigma_1)
(\prod_{1\leq i<j\leq n}A_{i, j}^{m_{ij}}\sigma_1)\\
&=&(\prod_{1\leq i<j\leq n}A_{i, j}^{m_{ij}})
(\sigma_1\prod_{1\leq i<j\leq n}A_{i, j}^{m_{ij}}\sigma_1^{-1})
\sigma_1^2\\
&=&(\prod_{1\leq i<j\leq n}A_{i, j}^{m_{ij}})
(\prod_{1\leq i<j\leq n}\sigma_1 A_{i, j}^{m_{ij}}\sigma_1^{-1})
A_{1, 2}\\
&=&A_{1, 2}^{2m_{12}+1}(\prod_{3\leq j\leq n}A_{1, j}^{m_{1j}+m_{2j}}
A_{2, j}^{m_{1j}+m_{2j}})
(\prod_{3\leq i<j\leq n} A_{i, j}^{2m_{ij}}), 
\end{eqnarray*}
and this element cannot be trivial because $2m_{12}+1$ cannot be $0$. Note 
that for the last equality we used the commutation relations 
\begin{eqnarray*}
&&\sigma _1A_{1, 2}\sigma _1^{-1}=A_{1, 2}, \\
&&\sigma _1A_{1, j}\sigma _1^{-1}=A_{2, j}, \;\;\;\forall \;\;j\geq 3, \\
&&\sigma _1A_{2, j}\sigma _1^{-1}=A_{1, j}, \;\;\;\forall \;\;j\geq 3, \\
&&\sigma _1A_{i, j}\sigma _1^{-1}=A_{i, j}, \;\;\;\forall \;\;3\leq i<j, 
\end{eqnarray*} 
which can be easily proved by using some of the formulae given above. 
\end{proof}
\begin{remark}
{\em As it is well-known (see \cite{brown}), any extension with abelian kernel 
corresponds to a 2-cocycle. In particular, the extension 
$1\rightarrow \mathfrak{P}_n\rightarrow PS_n\rightarrow S_n
\rightarrow 1$ corresponds to an element in 
$H^2(S_n, {\mathbb Z}^{\frac{n(n-1)}{2}})$. We illustrate this by computing 
explicitely the corresponding 2-cocycle for $n=3$.  
We consider the set-theoretical section $f:S_3\rightarrow PS_3$ defined by   
$f(1)=1$, $f(s_2)=\sigma _2$, $f(s_1)=
\sigma _1$, $f(s_1s_2)=\sigma _1\sigma _2$,  
$f(s_2s_1)=\sigma _2\sigma _1$ and    
$f(s_2s_1s_2)=\sigma _2\sigma _1\sigma _2$. The 2-cocycle afforded by this 
section is defined by 
$\;u:S_3\times S_3\rightarrow \mathfrak{P}_3,\;$ 
$u(x,y):=f(x)f(y)f(xy)^{-1}$, 
and a direct computation gives its explicit formula as in the following 
table (we have chosen here an additive notation for the abelian group 
$\mathfrak{P}_3\simeq {\mathbb Z}^3$):}
\begin{table}[h]
\begin{center}
\begin{tabular}{|l|r|c|c|c|c|c|}
 \hline
 & 1 & $s_2$ &$s_1$ & $s_1s_2$& $s_2s_1$& 
$s_2s_1s_2$\\ \hline
1 & 0& 0 &0& 0&0&0\\ \hline
$s_2$ & 0 & $A_{2, 3}$& 0&0&$A_{2, 3}$&$A_{2, 3}$\\ \hline
$s_1$ & 0 & 0& $A_{1, 2}$&$A_{1, 2}$& 0&$A_{1, 2}$\\ \hline
$s_1s_2$  & 0 & $A_{1, 3}$& 0&$A_{1, 2}$& $A_{1, 2}+A_{1, 3}$&
$A_{1, 2}+A_{1, 3}$\\ 
\hline
$s_2s_1$ & 0 & 0& $A_{1, 3}$&$A_{1, 3}+A_{2, 3}$& $A_{2, 3}$&
$A_{1, 3}+A_{2, 3}$\\ 
\hline
$s_2s_1s_2$ & 0 & $A_{1, 2}$& $A_{2, 3}$&$A_{1, 3}+A_{2, 3}$& 
$A_{1, 2}+A_{1, 3}$& 
$A_{1, 2}+A_{1, 3}+A_{2, 3}$\\ \hline
\end{tabular}
\caption{The 2-cocycle for n=3 associated to the section f}
\label{tab:2co}
\end{center}
\end{table}
\end{remark}
\section{$PS_n$ is linear} 
\setcounter{equation}{0}
${\;\;\;\;}$
In \cite{big}, \cite{krammer} it was proved  that the braid group $B_n$ 
is linear.  
More precisely, let $R$ be a commutative ring, $q$ and $t$ two invertible  
elements in $R$ and $V$ a free $R$-module of rank $n(n-1)/2$ with a 
basis $\{x_{i,j}\}_{1\leq i<j\leq n}$. Then the map   
$\rho:B_n\rightarrow  GL(V)$ defined by
\begin{eqnarray*}
&\sigma_kx_{k,k+1}=tq^2x_{k,k+1},& \\
&\sigma_kx_{i,k}=(1-q)x_{i,k}+qx_{i,k+1}, \; \; &i<k,\\
&\sigma_kx_{i,k+1}=x_{i,k}+tq^{k-i+1}(q-1)x_{k,k+1},\; \; &i<k,\\
&\sigma_kx_{k,j}=tq(q-1)x_{k,k+1}+qx_{k+1,j},\; \; &k+1<j,\\
&\sigma_kx_{k+1,j}=x_{k,j}+(1-q)x_{k+1,j},\; \; &k+1<j,\\
&\sigma_kx_{i,j}=x_{i,j},\; \; &i<j<k\;  or \;k+1<i<j,\\    
&\sigma_kx_{i,j}=x_{i,j}+tq^{k-i}(q-1)^2x_{k,k+1},\; \; &i<k<k+1<j, 
\end{eqnarray*}
and $\rho (x)(v)=xv$, for $x\in B_n$, $v\in V$,  
gives a representation of $B_n$, and if moreover  
$R={\mathbb R}[t^{\pm 1}]$ and $q\in {\mathbb R}\subseteq R$ with 
$0<q<1$, then the representation is faithful, see \cite{krammer}.

We consider now the general formula for $\rho$, in which we take $q=1$: 
\begin{eqnarray*}
&\sigma_kx_{k,k+1}=tx_{k,k+1},&\\
&\sigma_kx_{i,k}=x_{i,k+1}, \; \; &i<k,\\
&\sigma_kx_{i,k+1}=x_{i,k},\; \; &i<k,\\
&\sigma_kx_{k,j}=x_{k+1,j},\; \; &k+1<j,\\
&\sigma_kx_{k+1,j}=x_{k,j},\; \; &k+1<j,\\
&\sigma_kx_{i,j}=x_{i,j},\; \; &i<j<k\;  or \;k+1<i<j,\\    
&\sigma_kx_{i,j}=x_{i,j},\; \; &i<k<k+1<j.
\end{eqnarray*} 
One can easily see that these formulae imply  
\begin{eqnarray*}
\sigma_k^2 x_{k,k+1}&=&t^2x_{k,k+1},\\
\sigma_k^2 x_{i,j}&=&x_{i,j},\;\; if\;\; (i,j)\neq (k,k+1).
\end{eqnarray*}
One can then check that $\rho(\sigma_k^2)$ commutes with 
$\rho(\sigma_{k+1}^2)$, for all $1\leq k\leq n-2$,  
and so for $q=1$ it turns out that  
$\rho$ is a representation of $PS_n$.
\begin{theorem}  
This representation of $PS_n$ is faithful if $R={\mathbb R}[t^{\pm 1}]$. 
Henceforth, $PS_n$ is linear.
\end{theorem} 
\begin{proof} We first prove the following relations: 
\begin{eqnarray*}
A_{i,j}x_{i,j}&=&t^2x_{i,j}, \\
A_{i,j}x_{k,l}&=&x_{k,l},\;\; if\;\; (i,j)\neq (k,l).
\end{eqnarray*}
We do it by induction over $|j-i|$. If $|j-i|=1$ the relations follow from 
the fact that $A_{i,i+1}=\sigma_i^2$. Assume the relations hold for  
$|j-i|=s-1$. We want to prove them for $|j-i|=s$. We recall that  
$A_{i,j}=\sigma_{j-1}A_{i,j-1}\sigma_{j-1}^{-1}$, see (\ref{slab1}).  
We compute:  
\begin{eqnarray*}
A_{i,j}x_{i,j}&=&\sigma_{j-1}A_{i,j-1}\sigma_{j-1}^{-1}x_{i,j}\\
&=&\sigma_{j-1}A_{i,j-1}x_{i,j-1}\\
&=&\sigma_{j-1}t^2x_{i,j-1}\;\; (by \; induction)\\
&=&t^2x_{i,j}.
\end{eqnarray*}
On the other hand, if $(i,j)\neq(k,l)$ then 
$\sigma_{j-1}^{-1}x_{k,l}=x_{u,v}$ with  
$(i,j-1)\neq (u,v)$ and so we have:
\begin{eqnarray*}
A_{i,j}x_{k,l}&=&\sigma_{j-1}A_{i,j-1}\sigma_{j-1}^{-1}x_{k,l}\\
&=&\sigma_{j-1}A_{i,j-1}x_{u,v}\\
&=&\sigma_{j-1}x_{u,v}\;\; (by \; induction)\\
&=&\sigma_{j-1}\sigma_{j-1}^{-1}x_{k,l}\\
&=&x_{k,l}, \;\;\;q.e.d.
\end{eqnarray*}

To show that the representation is faithful take $b\in PS_n$ such that 
$\rho(b)=id_V$ and consider $\alpha (b)$, the image of $b$ in $S_n$.  
From the way $\rho$ is defined it follows  that 
\begin{eqnarray*}
bx_{i,j}=t^px_{\alpha (b)(i),\alpha (b)(j)}, \;\;\forall \;\;1\leq i<j\leq n, 
\end{eqnarray*}
with $p\in {\mathbb Z}$, where we made the convention $x_{r, s}:=x_{s, r}$ if 
$1\leq s<r\leq n$. 
Since $x_{i,j}$ is a basis in $V$ and we assumed $\rho (b)=id_V$, we obtain 
that the permutation $\alpha (b)\in S_n$ has the following property: if 
$1\leq i<j\leq n$ then either $\alpha (b)(i)=i$ and   
$\alpha (b)(j)=j$ or $\alpha (b)(i)=j$ and    
$\alpha (b)(j)=i$. Since we assumed $n\geq 3$, the only such permutation is 
the trivial one. Thus, we have obtained that $b\in Ker (\alpha )=
\mathfrak{P}_n$ and so we can write  
$b=\prod_{1\leq i<j\leq n}A_{i,j}^{m_{i,j}}$, with $m_{i, j}\in {\mathbb Z}$. 
By using the formulae given above for the action of $A_{i, j}$ on  
$x_{k, l}$ we immediately obtain 
\begin{eqnarray*}
bx_{k, l}=t^{2m_{k, l}}x_{k, l}, \;\;\forall \;\;1\leq k<l\leq n.  
\end{eqnarray*}
Using again the assumption $\rho (b)=id_V$, 
we obtain $t^{2m_{k, l}}=1$ and hence  
$m_{k, l}=0$ for all $1\leq k<l\leq n$, that is $b=1$, finishing 
the proof.
\end{proof}
\section{Pseudosymmetric groups and pseudosymmetric braidings}
\setcounter{equation}{0}
${\;\;\;\;}$
We recall (see \cite{k}, XIII.2) that to braid groups one can associate 
the so-called {\em braid category} ${\mathcal B}$, a universal braided 
monoidal category. Similarly, we can construct 
a pseudosymmetric braided category $\mathcal {PS}$ associated to 
pseudosymmetric groups.  Namely,  
the objects of $\mathcal {PS}$ are natural numbers $n\in {\mathbb N}$. 
The set of morphisms from $m$ to $n$ is empty if $m\neq n$ and is 
$PS_n$ if $m=n$. The monoidal structure of $\mathcal {PS}$ is defined as the 
one for ${\mathcal B}$, and so is the braiding, namely   
\begin{eqnarray*}
&&c_{n, m}:n\otimes m\rightarrow m\otimes n, \\
&&c_{0, n}=id_n=c_{n, 0}, \\
&&c_{n,m}=(\sigma_m\sigma_{m-1}\cdots \sigma_1)
(\sigma_{m+1}\sigma_{m}\cdots \sigma_2) 
\cdots (\sigma_{m+n-1}\sigma_{m+n-2}\cdots \sigma_n), \;\;if\;\;m, n>0.
\end{eqnarray*}
We denote by $t_{m,n}=c_{n,m}\circ c_{m,n}$ the double braiding. In view of 
Proposition \ref{carpseudosym}, in order to prove that $c$ is 
pseudosymmetric it is enough to check that, for all 
$m, n, p\in {\mathbb N}$ we have
\begin{eqnarray*}
&&(t_{m,n}\otimes id_p)\circ (id_m\otimes t_{n,p})=
(id_m\otimes t_{n,p})\circ (t_{m,n}\otimes id_p).
\end{eqnarray*}
This equality holds because 
it is a commutation relation between two elements in 
$\mathfrak{P}_{m+n+p}$.

Let ${\mathcal C}$ be a strict braided monoidal category with braiding $c$, 
let $n$ be a natural  
number and $V\in {\mathcal C}$. Consider the automorphisms 
$c_1, \cdots , c_{n-1}$ of $V^{\otimes \;n}$ defined by 
$c_i=id_{V^{\otimes \;(i-1)}}\otimes c_{V, V}\otimes 
id_{V^{\otimes \;(n-i-1)}}$. It is well-known (see \cite{k}) that there 
exists a unique group morphism $\rho _n^c:B_n\rightarrow 
Aut (V^{\otimes \;n})$ such that $\rho _n^c(\sigma _i)=c_i$, for all 
$1\leq i\leq n-1$. It is clear that, if $c$ is pseudosymmetric, then 
$\rho _n^c$ factorizes to a group morphism $PS_n\rightarrow 
Aut (V^{\otimes \;n})$. Thus, pseudosymmetric braided categories provide 
representations of pseudosymmetric groups.     

\end{document}